\documentclass{article}
\usepackage[utf8]{inputenc}

\usepackage{graphicx, tikz}
\usepackage{amsmath, amssymb, amsthm, mathrsfs, amsfonts}
\usepackage{dsfont}
\usepackage{adjustbox}
\usepackage{comment}
\usepackage{multicol}
\usepackage{bussproofs}
\usepackage{proof}
\usepackage{textcomp}
\usepackage{authblk}
\usepackage{framed}
\usepackage{xfrac}
\usepackage[numbers]{natbib}

\usepackage{hyperref}
\hypersetup{colorlinks=true,allcolors=blue}

\newtheorem{theorem}{Theorem}[section]
\newtheorem{lemma}[theorem]{Lemma}

\newtheorem{fact}[theorem]{Fact}

\theoremstyle{definition}
\newtheorem{definition}[theorem]{Definition}

\theoremstyle{remark}

\numberwithin{equation}{section}

\renewcommand{\t}{1}

\newcommand{\f}{0}

\newcommand{\CL}{{\sf CL}}

\newcommand{\X}{{\bf X}}
\newcommand{\xy}{\mathcal{C}}

\newcommand{\pe}[1]{{\color{black} #1}} 

\newcommand{\ds}[1]{{\color{black} #1}} 

\newcommand{\bdr}[1]{{\color{black} #1}} 

\makeatletter
\providecommand*{\Dashv}{%
  \mathrel{%
    \mathpalette\@Dashv\vDash
}
  }%
\newcommand*{\@Dashv}[2]{%
  \reflectbox{$\m@th#1#2$}%
}
\makeatother

\usepackage{fancyhdr}

\title{\textsc{On three-valued presentations of classical logic}\thanks{[This is the penultimate version, please use the published version in RSL for citation and reference]. We acknowledge Rohan French our Editor and two anonymous referees for detailed and very helpful comments. We are grateful to a referee for catching and allowing us to correct a mistake in a definition. Thanks also to the members of the Buenos Aires Logic Group for very helpful feedback on earlier versions of this paper, and to several audiences, in particular at the 2021 Navarra workshop on logical consequence held in Pamplona. This collaboration was initiated during a visit of DS to ENS Paris in 2020, and received support from the French programs ANR-17-EURE0017 (FrontCog) and ANR-19-CE28-0019-01 (Ambisense), and from PLEXUS (grant agreement n$^{\circ}$ 101086295), a Marie Sk\l odowska-Curie action funded by the EU under the Horizon Europe Research and Innovation Programme. BDR acknowledges a postdoctoral scholarship from CONICET during the time this paper was written. All four authors contributed substantially to the results, writing, and shape of this paper, and declare no conflict of interest.}}

\author{Bruno Da R\'e, Damian Szmuc, Emmanuel Chemla, Paul \'Egr\'e}

\date{}

\begin{document}

\maketitle

\begin{abstract}
   \noindent Given a three-valued definition of validity, which choice of three-valued truth tables for the connectives can ensure that the resulting logic coincides exactly with classical logic? We give an answer to this question for the five monotonic consequence relations $st$, $ss$, $tt$, $ss\cap tt$, and $ts$, when the connectives are negation, conjunction, and disjunction. For $ts$ and $ss\cap tt$ the answer is trivial (no scheme works), and for $ss$ and $tt$ it is straightforward (they are the collapsible schemes, in which the middle value acts like one of the classical values). For $st$, the schemes in question are the Boolean normal schemes that are either monotonic or collapsible.
\end{abstract}

\section{Characterizing classical logic}
\label{sec:intro}

Our goal in this paper is to provide a characterization of different ways in which classical logic can be presented in a three-valued setting. More precisely, our goal is to inventory which three-valued truth tables for negation, conjunction and disjunction can be paired with three-valued definitions of validity so as to yield exactly the same inferences that are obtained in more standard presentations of classical logic. While this project is mostly theoretical, it also has philosophical and conceptual motivations, about which we shall say more after stating the central results.

Toward our main goal, we first need to say more about the more standard ways in which classical logic has been characterized. Given a denumerable set of propositional variables $P=\{p, q, r, p', q', r',..\}$, a propositional logic \pe{is} a triple $\langle \mathcal{L}, C, \vdash\rangle$ such that $C$ is a finite set of $n$-ary connectives, $\mathcal{L}$ is the set of formulae generated from $P$ by application of the connectives in $C$, and $\vdash$ \pe{is} a relation between sets of formulae in $\mathcal{L}$. In this paper, we will mainly focus on the set $C=\{\neg, \vee, \wedge\}$, comprised of negation, disjunction and conjunction, forming a standard set of connectives in presentations of classical logic.

Given a propositional logic, what characterizes this logic as classical? One prominent answer to this question relies on two-valued semantics. On that view, a propositional logic is classical if the connectives are interpretable by specific two-valued truth functions, and $\vdash$ is interpretable by a specific relation between sets of truth values (where the values in question can be represented by $1$ and $0$, standing for True and False). Thus, for negation, conjunction, and disjunction to be classical, they must be interpretable by functions coextensional with $f_{\neg}(x)=1-x, f_{\vee}(x,y)=max(x,y), f_{\wedge}(x,y)=min(x,y)$ on the set $\mathcal{V}=\{1,0\}$. Moreover, $\vdash$ is classical provided $\Gamma\vdash \Delta$ if and only if for every valuation function $v$ which is a homomorphism from $(\mathcal{L},(\neg, \vee, \wedge))$ to $(\mathcal{V},(f_{\neg}, f_{\vee}, f_{\wedge}))$, $\{v(A) : A\in \Gamma\}\subseteq\{1\}$ implies $\{v(B) : B\in \Delta\}\cap\{1\}\neq \emptyset$. That is, for every valuation, the truth of the premises in $\Gamma$ implies the truth of some conclusion in $\Delta$, which we can write $\Gamma \vDash_2 \Delta$. 

But what justifies the choice of these tables, and of this definition of validity? 
One possibility to answer this question is to look at syntax, namely proof-theory. Of particular interest to us is Gentzen's perspective on the connectives and on the consequence relation. The leading idea behind Gentzen's approach in his seminal work \cite{gentzen} is that what makes a logic classical is the fact that the connectives and the consequence relation obey specific rules. The way Gentzen describes this is by specifying on the one hand \textit{structural rules} governing the consequence relation $\vdash$, and on the other \textit{operational rules} governing the connectives. Arguably, this perspective is more explanatory than the semantic perspective, because it tells us how inferences are shaped to begin with.

As structural rules, Gentzen proposed various properties such as reflexivity, monotonicity, contraction, and the Cut rule. 
For operational rules Gentzen proposed analytic rules, telling us how an inference involving a connective in premise position or in conclusion position depends on other inferences not involving that connective but only involving subformulae. For negation, conjunction, and disjunction, the rules of his calculus $\mathbf{LK}$ are as follows:

 \begin{figure}[ht]
\begin{center}
 \begin{tabular}{ccc}
 
 \AxiomC{$\Gamma\vdash A, \Delta$}
 \LeftLabel{}
 \UnaryInfC{$\Gamma, \neg A \vdash \Delta$}
 \DisplayProof

 &

 \AxiomC{$\Gamma, A, B\vdash \Delta$}
 \LeftLabel{}
 \UnaryInfC{$\Gamma, A\wedge B\vdash \Delta$}
 \DisplayProof

&

\AxiomC{$\Gamma,  A\vdash \Delta$}
 \AxiomC{$\Gamma , B \vdash \Delta$}
 \LeftLabel{}
 \BinaryInfC{$\Gamma, A\vee B \vdash  \Delta$}
 \DisplayProof

\\[2em]

 \AxiomC{$\Gamma, A \vdash \Delta$}
 \LeftLabel{}
 \UnaryInfC{$\Gamma \vdash \neg A, \Delta$}
 \DisplayProof
 
 &
 
\AxiomC{$\Gamma  \vdash A, \Delta$}
 \AxiomC{$\Gamma  \vdash B, \Delta$}
 \LeftLabel{}
 \BinaryInfC{$\Gamma  \vdash A\wedge B,  \Delta$}
 \DisplayProof

&

\AxiomC{$\Gamma\vdash A, B, \Delta$}
 \LeftLabel{}
 \UnaryInfC{$\Gamma\vdash A\vee B, \Delta$}
 \DisplayProof

\end{tabular}
\end{center}

 \caption{\textbf{LK} rules for negation (left), conjunction (middle), disjunction (right).}
 \label{fig:LK-rules}
 \end{figure}

Does it matter whether one starts from a proof-theoretic or from a semantic characterization of classical logic? One may say that it does not, considering that the semantic and the syntactic perspective can be made to coincide. Gentzen's sequent calculus $\mathbf{LK}$, which characterizes $\vdash$ syntactically, is sound and complete for the semantic interpretation $\vDash_2$ of the consequence relation between sets of formulae.
However, Gentzen \cite{gentzen} has shown that Cut is eliminable from $\mathbf{LK}$. In other words, the set of provable sequents in $\mathbf{LK}$ minus Cut is the same as the set of provable sequents in $\mathbf{LK}$, i.e., exactly those that are classically valid.
So, in the same way in which there is a different syntactic characterization of classical logic than the one based on $\mathbf{LK}$, one can also ask if there can be different semantic characterizations of classical logic beside the two-valued approach. 

As it turns out, some authors have provided various three-valued semantics for classical logic. For instance, Girard in \cite{girard1987proof} offers a non compositional semantics based on three-valued valuations (the so-called {Sch\"utte valuations}), while Cobreros et al. \cite{cobreros2012tolerant} do the same using the Strong Kleene valuations, and more recently, Szmuc and Ferguson \cite{szmucferguson} \bdr{and Ferguson \cite{ferguson2022monstrous}} show that the Weak Kleene valuations also work. All of these characterizations are given by the so-called $st$-consequence relation, defined by the fact that when all premises in $\Gamma$ take the value $1$ in the set $\{1, \sfrac{1}{2}, 0\}$, some conclusion in $\Delta$ takes a value other than $0$---see \cite{frankowski} and \cite{blasio2017} for related discussions of this notion of consequence.

In \cite{chemla2019many}, it is shown that beside $st$, other substructural consequence relations admit connectives satisfying Gentzen's operational rules, and are representable by means of three-valued operators. A case of interest is the non-reflexive relation $ts$, defined by the fact that when all premises in $\Gamma$ take a value other than $0$, some conclusion takes the value $1$---see \cite{malinowski}, \cite{French2016} and \cite{NicolaiRossi2018}. Moreover, \cite{chemla2019many} shows that when the language contains constants for the truth values, $ts$ admits as Gentzen-regular connectives the same Strong Kleene negation, conjunction and disjunction as $st$. Similar results, both positive and negative, are obtained for alternative definitions of logical consequence, in particular \pe{for} $ss$ (preservation of the value 1 from premises to conclusion), \pe{for} $tt$ (preservation of non-falsity) and \pe{for} their intersection $ss\cap tt$.\footnote{\ds{The consequence relation $ss \cap tt$ over the Strong Kleene valuations renders the well-known logic $\mathrm{RM}_{fde}$, i.e., the first-degree entailment fragment of the relevant logic $\mathrm{R}$-mingle, as well documented and discussed in \cite{dunn1976b}.}}

The results in \cite{chemla2019many}, however, did not purport to give a trivalent characterization of \textit{classical} logic as defined above, namely in terms of both structural and operational rules. Instead, they focus only on operational rules, and for the most part they assume that the language can express all truth values, including the third value $\sfrac{1}{2}$. Given these results, we are led to the following more general question: {\it what are all the three-valued schemes that can be used to characterize exactly those inferences valid in the two-valued presentation of classical logic, i.e., $\vDash_2$?} To answer this question, we determine, for the five definitions of semantic validity mentioned above ($st$, $ss$, $tt$, $ss\cap tt$, and $ts$), which three-valued truth tables can be assigned to negation, conjunction, and disjunction so as to yield all and only the inferences of {the two-valued presentation of} classical logic. For $ts$, the answer is trivial: no scheme will work, since $p\vdash p$ fails in $ts$---see, e.g., \cite{malinowski} and \cite{ cobreros2012tolerant}. For the remaining four definitions of validity, the answer is less obvious, in particular in the case of $st$. For $st$, it is a contested matter whether it supports classical \textit{meta-inferences} such as the Cut rule---see, e.g., \cite{cobreros2012tolerant} and \cite{dicher2019st}. However, here we are interested primarily in whether $st$ can support the same classical \textit{inferences} as two-valued semantics.

Our work proceeds as follows: in Section \ref{sec:definitions} we start by a review of three-valued definitions of validity, with an indication of the valuation schemes playing a central role in our results. Section \ref{sec:main} presents our main results, whose proof we defer to the Appendix to ease reading. Section \ref{sec:discussion} concludes with comparisons and a discussion of the philosophical value of those results.

\section{Definitions}\label{sec:definitions}

 The question we are investigating in this paper can be put as follows: given a three-valued definition of logical consequence $\xy$, what set of truth tables $\X$ (or scheme) for the connectives can be such as to ensure that the resulting consequence relation $\vDash^{\xy}_{\X}$ coincides with classical consequence. In this section  we first introduce the five notions of validity of interest in a trivalent setting, where the truth values are going to be 1, $\sfrac{1}{2}$, and 0. We then define the relevant properties of connectives and their truth tables, and give an overview of the way in which these properties constrain classicality for different consequence relations.
 
\subsection{Logical Consequence}

This section introduces five entailment relations corresponding to distinct ways of thinking of validity in a three-valued setting. They include the so-called pure, mixed and intersective definitions of logical consequence, as defined in \cite{chemla2017characterizing}. 

\begin{definition} 
Let a valuation be a function $v$ from formulae to the set $\{ 1, \sfrac{1}{2}, 0 \}$.
\bdr{
 \begin{description}
 
   \item[($ss$-validity)] $\Gamma \vDash^{ss}\Delta$ if and only if for every valuation $v$,\\ if $v(A)= 1$ for every $A \in \Gamma$, then $v(B)= \t$ for some $B \in \Delta$.

   \item[($tt$-validity)] $\Gamma \vDash^{tt}\Delta$ if and only if for every valuation $v$,\\ if $v(A) \in \{\t, \sfrac{1}{2}\}$ for every $A \in \Gamma$, then $v(B) \in \{\t, \sfrac{1}{2}\}$ for some $B \in \Delta$.

   \item[($st$-validity)] $\Gamma \vDash^{st}\Delta$ if and only if for every valuation $v$,\\ if $v(A)= 1$ for every $A \in \Gamma$, then $v(B) \in \{\t, \sfrac{1}{2}\}$ for some $B \in \Delta$.
   
   \item[($ts$-validity)] $\Gamma \vDash^{ts}\Delta$ if and only if for every valuation $v$,\\ if $v(A) \in \{\t, \sfrac{1}{2}\}$ for every $A \in \Gamma$, then $v(B) = 1$ for some $B \in \Delta$.

    \item[($ss \cap tt$-validity)] $\Gamma \vDash^{ss\cap tt}\Delta$ if and only if $\Gamma \vDash^{ss}\Delta$ and 
            $\Gamma \vDash^{tt}\Delta$;\\ or equivalently, if and only if for every valuation $v$,\\$\inf\{v(A)|A \in \Gamma\} \leq_{\mathbb{Q}} \sup \{v(B)|B \in \Delta\}$, where $\leq_{\mathbb{Q}}$ is the usual order over the rational numbers.
            
         \end{description}

}

\end{definition}

Basically, the pure notions of validity are the ones definable in terms of the preservation of a fixed set of designated values between premises and conclusions, they include $ss$ (preservation of value 1) and $tt$ (preservation of values that are not 0). The mixed notions of validity $st$ and $ts$ define logical consequence not in terms of preservation but in terms of specific constraints between values that can differ for premises and conclusions (not going from truth to falsity for $st$, or from non-falsity to non-truth for $ts$). Finally, the intersective notion of validity $ss\cap tt$  has also been called order-theoretic in \cite{field2008, chemla2017characterizing}, because it is equivalent to requiring that, relative to the total ordering of truth-values $0<\sfrac{1}{2}<1$, the largest value of the conclusions should not be smaller than the smallest value of the premises. 

Although more entailment relations are conceivable, in \cite{chemla2017characterizing} these five were identified as the so-called {\it intersective mixed} consequence relations.\footnote{\bdr{These consequence relations are called \textit{intersective mixed consequence relations} in \cite{chemla2017characterizing} because they are all the consequence relations definable as intersections between mixed consequence relations (which include pure consequence relations as defined in \cite{chemla2017characterizing}). From the lattice displayed in Figure ~\ref{consequence-relations} notice that $ss \cap tt$ is the only intersective consequence relation which is not a pure or a mixed consequence relation. Given the inclusion between the logics, the other consequence relations are all the possible intersections between mixed and pure consequence relations. See \cite{chemla2017characterizing} for more details.}} They form a natural class by corresponding to the three-valued \textit{monotonic} consequence relations \pe{(namely such that if $\Gamma \vdash \Delta$, then $\Gamma,\Gamma' \vdash \Delta,\Delta'$)}.
These consequence relations are related as depicted in Figure~\ref{consequence-relations}, in which a lower relation is an extensional subset of a higher relation.

  {\small 
\begin{figure}[h]
\centering
\begin{tikzpicture}
  \node (max) at (0,2) {$st$};
  \node (a) at (-2,0) {$ss$};
  \node (c) at (2,0) {$tt$};
  \node (d) at (0,-2) {$ss \cap tt$};
  \node (min) at (0,-4) {$ts$};
  
  \draw (max) -- (a);
  \draw (max) -- (c);
  \draw (a) -- (d);
  \draw (c) -- (d);
  \draw (d) -- (min);

\end{tikzpicture}
\caption{The five intersective mixed consequence relations}
\label{consequence-relations}
\end{figure}
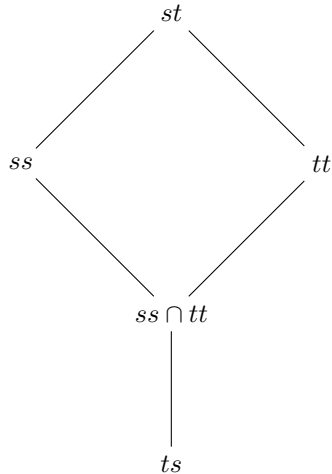
}

\subsection{Schemes for the connectives: Boolean normal, Monotonic, Collapsible}

We define a \textit{three-valued valuation scheme} $\X$ as a triple $(f_\neg, f_\wedge, f_\vee)$ of operations, namely of three-valued truth tables for the connectives. The properties of a scheme are defined in terms of the properties of its operations. Here we single out three main properties of interest: Boolean normality, monotonicity, and collapsibility.\bigskip

We first define \textit{Boolean normal} operations, that is operations that behave on Boolean values like their corresponding (``normal'') counterpart in classical logic. \pe{This property is also referred to in the literature as \textit{normality} (\cite{priest2008}), or as \textit{regularity} (\cite{priest2022interpretation}).
For more on the origin of this terminology, going back to \cite{post1921} and \cite{rescherbook}, see \cite{Teijeiro2022} and references therein}.\bdr{\footnote{\bdr{The authors in \cite[p.129]{carnielli2000formal} have introduced a related property which they called \textit{hyper-classicality} which they defined as follows: ``a three-valued matrix is hyper-classical if the restriction of its associated function to the classical domain (values 1 and 0) will have its image in the classical codomain (values 1 and 0)''. According to this definition, all Boolean normal schemes are hyper-classical.}}}

\begin{definition}[Boolean normality]
\label{def:boolean-normality}
An $n$-ary operation $\star$ is \textit{Boolean normal} if and only if for $\{a_1,...,a_n\} \subseteq \{0,1\}$, $\star(a_1,...,a_n)=\star^{\CL}(a_1,...,a_n)$, where $\star^{\CL}$ is the corresponding operation over the usual two-element Boolean algebra. 
A scheme is \textit{Boolean normal} iff each of its operations is.
\end{definition}

\bdr{
Next, we assume that truth values are ordered with regard to $\leq_{\rm I}$ in terms of their so-called informational value, as described in \cite{fitting1994}, that is: $\sfrac{1}{2} <_{\rm I} 0$ and $\sfrac{1}{2} <_{\rm I} 1$, as depicted in Figure \ref{fig:inforder}. Given such an ordering relation $<_{\rm I}$, we can define the componentwise ordering based on this order as follows: $\langle a_1,...,a_n \rangle \leq_{\rm I}^{comp} \langle b_1,...,b_n \rangle$ if and only if $a_j \leq_{\rm I} b_j$ for all $1 \leq_{\rm I} j \leq _{\rm I} k$.

\begin{definition}[Monotonicity]
\label{def:monotonicity}
An $n$-ary operation $\star$ is (upward) \textit{monotonic} if and only if whenever $\langle a_1,...,a_n \rangle \leq_{\rm I}^{comp} \langle b_1,...,b_n \rangle$ then $\star(a_1,...,a_n) \leq_{\rm I} \star(b_1,...,b_n)$. 
A scheme is monotonic if and only if each of its operations is.
\end{definition}

}

\begin{figure}[ht]
\centering
{\small 
\begin{tikzpicture}
  \node (max) at (0,0) {${\sfrac{1}{2}}$};
  \node (a) at (-1,1) {$1$};
  \node (c) at (1,1) {$0$};
  \draw (a) -- (max) -- (c);
\end{tikzpicture}
}
\caption{Hasse diagram of the information order $\leq_{\rm I}$}\label{fig:inforder}
\end{figure}
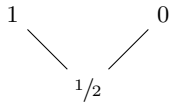

In 2D truth table format, monotonic operations are such that no two distinct classical values are found next to one another, horizontally or vertically. This is easy to prove and we refer to Appendix \ref{app:mon} for a tighter characterization.

By combining Boolean normality and monotonicity, we obtain the Boolean normal monotonic operations for negation, conjunction and disjunction presented in the truth tables in Figure~\ref{boolean-normal-monotonic}. Here and elsewhere, when a cell contains more than one value, this means that any choice of a value renders an operation with the desired properties, independently of the choice of values in other cells (in this case, Boolean normality and monotonicity). 

\begin{figure}[ht]
\centering
\begin{tabular}{cccc}
\begin{tabular}{c|c}
				 
				& $\neg$    \\ \hline
				$\t$ & $\f$  \\
				$\sfrac{1}{2}$ & $\sfrac{1}{2}$\\
				$\f$ & $\t$  \\  
			\end{tabular}
 
			&
			
			\begin{tabular}{c|ccc}
				 
	$\wedge$ & $\t$ & $\sfrac{1}{2}$ & $\f$ \\   \hline
	$\t$ & $\t$ & $\sfrac{1}{2}$ & $\f$ \\ 
	$\sfrac{1}{2}$ & $\sfrac{1}{2}$ & $\sfrac{1}{2}$ & $\sfrac{1}{2}, \f$ \\  
	$\f$ & $\f$ & $\sfrac{1}{2}, \f$ & $\f$ \\  
			\end{tabular}
			
			&
			
				\begin{tabular}{c|ccc}
				 
				$\vee$ & $\t$ & $\sfrac{1}{2}$ & $\f$ \\   \hline
				$\t$ & $\t$ & $\t, \sfrac{1}{2}$ & $\t$ \\ 
				$\sfrac{1}{2}$ & $\t, \sfrac{1}{2}$ & $\sfrac{1}{2}$ & $\sfrac{1}{2}$ \\  
				$\f$ & $\t$ & $\sfrac{1}{2}$ & $\f$ \\

			\end{tabular}

			\end{tabular}
\caption{All the Boolean normal monotonic schemes}
\label{boolean-normal-monotonic}
\end{figure}

To introduce our last relevant property, we define ``$\alpha$-collapsers'', operations $\tau_\alpha$ defined as $\tau_\alpha(0)=0$, $\tau_\alpha(1)=1$, and $\tau_\alpha(\sfrac{1}{2})=\alpha$, for $\alpha=0$ or $\alpha=1$. As can be seen, collapsers preserve Boolean values, and collapse the third value onto $\alpha$.

\begin{definition}[Collapsibility]
An $n$-ary operation $\star$ is an $\alpha$-collapsible version of a classical operation $\star^{\CL}$ iff $\tau_\alpha(\star(x_1, \dots, x_n)) = \star^{\CL}(\tau_\alpha(x_1), \dots, \tau_\alpha(x_n))$. A scheme is $\alpha$-collapsible if and only if all of the operations are $\alpha$-collapsible.
\end{definition}

In terms of truth tables, the $1$-collapsible (henceforth, truth-collapsible) and the $0$-collapsible (falsity-collapsible) operations are as reported in Figures~\ref{truth-collapsible}~and~\ref{falsity-collapsible}, respectively.

\begin{figure}[h]
\centering
\begin{tabular}{ccc}
\begin{tabular}{c|c}
               & $\neg$ \\ \hline
$\t$           & $\f$   \\
$\sfrac{1}{2}$ & $\f$   \\
$\f$           & $\t,\sfrac{1}{2}$  
\end{tabular}
&

\begin{tabular}{c|ccc}
$\wedge$       & $\t$ & $\sfrac{1}{2}$ & $\f$ \\ \hline
$\t$           & $\t,\sfrac{1}{2}$               & $\t, \sfrac{1}{2}$ & $\f$ \\
$\sfrac{1}{2}$ & $\t, \sfrac{1}{2}$ & $\t, \sfrac{1}{2}$ & $\f$ \\
$\f$           & $\f$               & $\f$               & $\f$
\end{tabular}
&
\begin{tabular}{c|ccc}
$\vee$         & $\t$ & $\sfrac{1}{2}$ & $\f$ \\ \hline
$\t$           & $\t,\sfrac{1}{2}$               & $\t, \sfrac{1}{2}$ & $\t,\sfrac{1}{2}$               \\
$\sfrac{1}{2}$ & $\t, \sfrac{1}{2}$ & $\t, \sfrac{1}{2}$ & $\t, \sfrac{1}{2}$ \\
$\f$           & $\t,\sfrac{1}{2}$               & $\t, \sfrac{1}{2}$ & $\f$ 
\end{tabular}
\end{tabular}
\caption{All the truth-collapsible schemes}\label{truth-collapsible}
\end{figure}

\begin{figure}[h]
\centering
\begin{tabular}{ccc}
\begin{tabular}{c|c}
               & $\neg$ \\ \hline
$\t$           & $\sfrac{1}{2},\f$   \\
$\sfrac{1}{2}$ & $\t$   \\
$\f$           & $\t$  
\end{tabular}
&
\begin{tabular}{c|ccc}
$\wedge$       & $\t$ & $\sfrac{1}{2}$ & $\f$ \\ \hline
$\t$           & $\t$               & $\sfrac{1}{2}, \f$ & $\sfrac{1}{2},\f$               \\
$\sfrac{1}{2}$ & $\sfrac{1}{2}, \f$ & $\sfrac{1}{2}, \f$ & $\sfrac{1}{2}, \f$ \\
$\f$           & $\sfrac{1}{2},\f$               & $\sfrac{1}{2}, \f$ & $\sfrac{1}{2},\f$      
\end{tabular}
&
\begin{tabular}{c|ccc}
$\vee$         & $\t$ & $\sfrac{1}{2}$ & $\f$ \\ \hline
$\t$           & $\t$ & $\t$               & $\t$               \\
$\sfrac{1}{2}$ & $\t$ & $\sfrac{1}{2}, \f$ & $\sfrac{1}{2}, \f$ \\
$\f$           & $\t$ & $\sfrac{1}{2}, \f$ & $\sfrac{1}{2},\f$              
\end{tabular}
\end{tabular}
\caption{All the falsity-collapsible schemes}\label{falsity-collapsible}
\end{figure}

We can see how these translates the definitions. First, the Boolean corners of the table should yield the same values as the corresponding Boolean operations, up to $\tau_\alpha$. Second, in an area in which one can move by applying $\tau_\alpha$ to one or both of the inputs, all output values should be the same, again up to $\tau_\alpha$. Therefore, the truth-collapsible scheme is one in which the values $1$ and $\sfrac{1}{2}$ play the same functional role, whereas in the falsity-collapsible case the values $0$ and $\sfrac{1}{2}$ play the same role. Figures~\ref{boolean-normal-truth-collapsible}~and~\ref{boolean-normal-falsity-collapsible} display the \textit{Boolean normal collapsible operations}.
\begin{figure}[ht]
\centering
\begin{tabular}{ccc}
\begin{tabular}{c|c}
               & $\neg$ \\ \hline
$\t$           & $\f$   \\
$\sfrac{1}{2}$ & $\f$   \\
$\f$           & $\t$  
\end{tabular}
&

\begin{tabular}{c|ccc}
$\wedge$       & $\t$ & $\sfrac{1}{2}$ & $\f$ \\ \hline
$\t$           & $\t$               & $\t, \sfrac{1}{2}$ & $\f$ \\
$\sfrac{1}{2}$ & $\t, \sfrac{1}{2}$ & $\t, \sfrac{1}{2}$ & $\f$ \\
$\f$           & $\f$               & $\f$               & $\f$
\end{tabular}
&
\begin{tabular}{c|ccc}
$\vee$         & $\t$ & $\sfrac{1}{2}$ & $\f$ \\ \hline
$\t$           & $\t$               & $\t, \sfrac{1}{2}$ & $\t$               \\
$\sfrac{1}{2}$ & $\t, \sfrac{1}{2}$ & $\t, \sfrac{1}{2}$ & $\t, \sfrac{1}{2}$ \\
$\f$           & $\t$               & $\t, \sfrac{1}{2}$ & $\f$ 
\end{tabular}
\end{tabular}
\caption{All the Boolean normal truth-collapsible schemes}\label{boolean-normal-truth-collapsible}
\end{figure}

\begin{figure}[ht]
\centering
\begin{tabular}{ccc}
\begin{tabular}{c|c}
               & $\neg$ \\ \hline
$\t$           & $\f$   \\
$\sfrac{1}{2}$ & $\t$   \\
$\f$           & $\t$  
\end{tabular}
&
\begin{tabular}{c|ccc}
$\wedge$       & $\t$ & $\sfrac{1}{2}$ & $\f$ \\ \hline
$\t$           & $\t$               & $\sfrac{1}{2}, \f$ & $\f$               \\
$\sfrac{1}{2}$ & $\sfrac{1}{2}, \f$ & $\sfrac{1}{2}, \f$ & $\sfrac{1}{2}, \f$ \\
$\f$           & $\f$               & $\sfrac{1}{2}, \f$ & $\f$      
\end{tabular}
&
\begin{tabular}{c|ccc}
$\vee$         & $\t$ & $\sfrac{1}{2}$ & $\f$ \\ \hline
$\t$           & $\t$ & $\t$               & $\t$               \\
$\sfrac{1}{2}$ & $\t$ & $\sfrac{1}{2}, \f$ & $\sfrac{1}{2}, \f$ \\
$\f$           & $\t$ & $\sfrac{1}{2}, \f$ & $\f$              
\end{tabular}
\end{tabular}

\caption{All the Boolean normal falsity-collapsible schemes}\label{boolean-normal-falsity-collapsible}
\end{figure}

Finally, notice that no collapsible negation is monotonic, because the third value yields a determinate value for collapsible negations, and an indeterminate for the monotonic negation. This implies that no collapsible scheme is monotonic, and conversely. 

Various examples from the literature can be given to illustrate those schemes. The well-known Strong Kleene scheme and the Bochvar/Weak Kleene scheme are both Boolean normal monotonic schemes. Boolean normal monotonic schemes also include other schemes, such as \ds{the scheme characteristic of Lisp logic as discussed in Fitting's \cite{fitting1994}, first introduced by McCarthy in \cite{McCarthy}---also to be found in the presupposition projection literature, in particular in Peters' \cite{peters1979truth}.\footnote{\pe{This scheme can be viewed as a compromise between a Strong Kleene and a Weak Kleene scheme in that it is asymmetric: binary operations are understood as Weak Kleene on their first argument, and Strong Kleene on the second.}}} Likewise, the collapsible schemes are not just theoretical possibilities: an example of truth-collapsible scheme can be found in Cantwell's \cite{cantwell2008logic}, under the name ``non-bivalent classical valuation''. Cantwell gives tables for negation, conjunction, and disjunction that are Boolean normal truth-collapsible. Additionally, he defines a conditional operator (originally introduced independently by Cooper in \cite{cooper1968propositional}), which is not Boolean normal (it yields the value $\sfrac{1}{2}$ when the antecedent has value $0$, and takes the value of the consequent otherwise), but which could be shown to be truth-collapsible.\footnote{\ds{For visually-inclined readers, we include Cantwell's truth-tables below:

\medskip

{\scriptsize
\begin{tabular}{cccc}
\begin{tabular}{c|c}
               & $\neg$ \\ \hline
$\t$           & $\f$   \\
$\sfrac{1}{2}$ & $\f$   \\
$\f$           & $\t$  
\end{tabular}
&
\begin{tabular}{c|ccc}
$\wedge$       & $\t$ & $\sfrac{1}{2}$ & $\f$ \\ \hline
$\t$           & $\t$               & $\sfrac{1}{2}$ & $\f$ \\
$\sfrac{1}{2}$ & $\sfrac{1}{2}$ & $\sfrac{1}{2}$ & $\f$ \\
$\f$           & $\f$               & $\f$               & $\f$
\end{tabular}
&
\begin{tabular}{c|ccc}
$\vee$         & $\t$ & $\sfrac{1}{2}$ & $\f$ \\ \hline
$\t$           & $\t$               & $\t$ & $\t$               \\
$\sfrac{1}{2}$ & $\t$ & $\sfrac{1}{2}$ & $\sfrac{1}{2}$ \\
$\f$           & $\t$               & $\sfrac{1}{2}$ & $\f$ 
\end{tabular}
&
\begin{tabular}{c|ccc}
$\rightarrow$         & $\t$ & $\sfrac{1}{2}$ & $\f$ \\ \hline
$\t$           & $\t$               & $\sfrac{1}{2}$ & $\f$               \\
$\sfrac{1}{2}$ & $\t$ & $\sfrac{1}{2}$ & $\f$ \\
$\f$           & $\sfrac{1}{2}$               & $\sfrac{1}{2}$ & $\sfrac{1}{2}$ 
\end{tabular}
\end{tabular}
}}
}

\subsection{Interaction of logical consequence and schemes: Overview of the results}

With the definition of logical consequence and of a valuation scheme in hand,
we can  restate our main goal more precisely as follows.

\begin{definition}

Given a scheme $\X$ and a definition of logical consequence $\xy$, we write $\vDash_{\X}^{\xy}$ the corresponding \textit{consequence relation}, namely the set of valid arguments based on the scheme $\X$ relying on the $\xy$ definition of validity.

\end{definition}

\noindent Our key question is: for a given definition of validity $\xy$, what scheme is inferentially classical? The relevant definition of inferential classicality is as follows:

\begin{definition}

Given a scheme $\X$ and a definition of logical consequence $\xy$, we say $\vDash_{\X}^{\xy}$ is \textit{inferentially classical} if and only if for every pair of sets of formulae $\Gamma, \Delta$, we have $\Gamma \vDash_{\X}^{\xy} \Delta$ if and only if $\Gamma \models_2 \Delta$.

\end{definition}

Before stating the main result of this paper, we justify the choice of the properties of the schemes highlighted above. Boolean normality provides an upper bound for classicality: {for every} consequence relation, it ensures that the arguments it supports involving negation, conjunction, and disjunction are a subset of the classical arguments (see Lemma \ref{lemma:bn}), and is furthermore a necessary property for this to hold with many consequence relations (see Lemma \ref{lemma:normal}). Monotonicity, for the specific case of $st$, provides a lower bound: it ensures that the classical inferences are a subset of the ones supported (Lemma \ref{lemma:monotonic}). Collapsibility, finally, provides either a lower bound or an upper bound, depending on which consequence is considered (Lemmas \ref{lemma:truthcollapsible}, \ref{lemma:falsitycollapsible}, \ref{lemma:truth-collapsible and tt}, \ref{lemma:falsity-collapsible and ss}).

\section{Main characterization results}\label{sec:main}

With these ingredients in place we are ready to present the main results of this paper. The results fall in two main classes : each of the consequence relations $st, ss, tt$ supports a positive characterization of classical logic ; $ss\cap tt$ and $ts$, on the other hand, fail to support classical logic for any scheme. We start with the presentation of those negative results, for which the explanation is straightforward, reading Figure \ref{consequence-relations} bottom up.  

\subsection{Negative results: $ts$ and $ss \cap tt$}

As is well-known, the consequence relation $ts$ is nonreflexive, hence no scheme can combine with it to make it classical.

\begin{theorem}\label{theorem:ts}
$\vDash^{ts}_{\X} \ \neq \ \vDash_{2}$, for every three-valued scheme $\X$.
\end{theorem}

\begin{proof}
For an atomic proposition $p$, independently of $\X$, $p \vDash^{ts}_{\X} p$ does not hold, while $p \vDash_{2} p$ does.
\end{proof}

In Section~\ref{sec:gr} and in Appendix~\ref{app:gr} we will show that inferential classicality can be obtained inductively from two parts, essentially distinguishing the role of formulae with and without connectives: (i)~some structural properties for atomic propositions (namely reflexivity), (ii)~some Gentzen regularity for the connectives. Here, with $ts$ we show how the first condition is broken and prevents inferential classicality, independently of the connectives.

From a structural point of view, $ss\cap tt$ is a Tarskian relation, unlike $ts$: it is reflexive, monotonic, and transitive. Despite that, it fails to support classical logic. As the following result shows, it cannot support both the Law of Excluded Middle and the principle of Explosion in a way that makes negation coherent.

\begin{theorem}
$\vDash^{ss\cap tt}_{\X} \ \neq \ \vDash_{2}$, for every three-valued scheme $\X$.
\end{theorem}

\begin{proof}
First, consider a formula $p$ and a valuation $v$ in which $v(p)=\sfrac{1}{2}$. For the classical inference $\vDash_{2} p, \neg p$ to $ss \cap tt$-hold, it must be that $v(\neg p)=1$, that is $\neg(\sfrac{1}{2})=1$. Second, consider atomic formulae $p$, $q$ and a valuation $v$ in which $v(p)=\sfrac{1}{2}$ and $v(q)=0$. For the classical inference $p, \neg p \vDash_{2} q$ to $ss \cap tt$-hold, it must be that $\neg(\sfrac{1}{2})=0$. Contradiction.
\end{proof}

\noindent This result is closely related to Theorem~4.3 of \cite{chemla2019many}, showing that $ss\cap tt$ admits no Gentzen-regular negation. The result holds even when the consequence relation is restricted to single conclusions. To validate explosion, the negation of $1$ and of $\sfrac{1}{2}$ must be $0$. To satisfy the entailment from $p$ to $\neg \neg p$ the negation of $0$ must be different from $0$ when $p$ is valued to $\sfrac{1}{2}$. To satisfy the converse entailment from $\neg \neg p$ to $p$, the negation of $0$ cannot be 1, so must be $\sfrac{1}{2}$. But then when $p$ is valued to $1$, $\neg \neg p$ is valued to $\sfrac{1}{2}$, so $p$ cannot entail $\neg \neg p$ in all cases.

A simple takeaway from this result is that when entertaining the $ss \cap tt$ definition of logical consequence there isn't a three-valued scheme ${\bf X}$ that supports the same valid inferences as the two-valued presentation of classical logic, mainly because there isn't a truth table for negation that supports the same valid inferences in that respect. But the problem isn't restricted to negation: as shown in \cite{chemla2019many} other connectives also cannot find an appropriate truth table so as to validate the intended inferences, for example, the material conditional. Interestingly enough, in \cite{chemla2019many} it is shown that some connectives (like conjunction and disjunction) do indeed have compatible truth tables that validate the target inferences---at least in the restricted language where only those connectives are featured. With these reflections, we hope to shed some light on the aforementioned impossibility regarding $ss \cap tt$, by making the appropriate qualifications.\footnote{One may wonder whether this result shows that ``negative'' or  negation-related connectives cannot be supported by any three-valued truth table when the $ss \cap tt$ consequence relation is around, but ``positive'' or non-negation-related connectives can. A discussion of this is far beyond the scope of this paper, but we hope to elaborate on this in further research.}

\bdr{Before turning to the positive results, it is important to mention that the negative results presented in this section can be easily generalized to many-valued semantics with more than three values, since the proofs of these statements appear to be independent of the number of nonclassical values. The natural requirement for this generalization is that the $ts$-consequence relation defined for this many-valued semantics be nonreflexive and that the $ss\cap tt$ be such that $ss$ lacks tautologies and $tt$ lacks \pe{logical contradictions}.}

\subsection{Positive results: $ss$, $tt$, and $st$}\label{sec:positive}

The fact that $ss$ and $tt$ can support classical logic separately follows from the simple fact that the value $\sfrac{1}{2}$ can be made to mirror the role of either $0$ or $1$ in a given scheme. This is the sense in which collapsibility (whether for falsity, or truth) yields classical logic. 

\begin{theorem}\label{thm:ss}
Let $\X$ be a three-valued scheme. $\vDash^{ss}_{\X} \ = \ \vDash_{2}$ if and only if $\X$ is falsity-collapsible (see Fig.~\ref{falsity-collapsible}).
\end{theorem}

\begin{proof}
See Theorems~\ref{theorem:sufficient-conditions-ss}~and~\ref{theorem:necessary-conditions-ss} in Appendix.
\end{proof}

\begin{theorem}\label{thm:tt}
Let $\X$ be a three-valued scheme. $\vDash^{tt}_{\X} \ = \ \vDash_{2}$ if and only if $\X$ is truth-collapsible (see Fig.~\ref{truth-collapsible}).
\end{theorem}

\begin{proof}
See Theorems~\ref{theorem:sufficient-conditions-tt}~and~\ref{theorem:necessary-conditions-tt} in Appendix.
\end{proof}

One direction to those two results---the one stating the sufficient conditions---is not surprising, arguably: $ss$ and $tt$ are pure consequence relations, i.e. can be formulated as preservation of some set of values, usually called \textit{designated values}. In this sense, the set of designated values that characterizes $ss$-validity consists in the singleton $\{1\}$, while $tt$-validity can be characterized as preserving the values on the set $\{ 1, \sfrac{1}{2} \}$. If we think of designated values as representing truth and undesignated values as representing falsity, the results above are foreseeable. In $ss$, the intermediate value doesn't belong to the set of designated values: that is why the falsity-collapsible schema works. On the other hand, in $tt$, the intermediate value belongs to the set of designated values, and in this case, the truth-collapsible schema works. However, the other direction of these results---the one stating the necessary conditions---is more surprising in that no other schemes than those collapsible work in the intended way.

The case of $st$ is the least straightforward among the five trivalent consequence relations examined here. For this scheme we get a disjunctive characterization involving collapsibility and monotonicity as separate conditions. One way in which this may be understood is by looking at negation first: when negation is monotonic, the value $\sfrac{1}{2}$ cannot be interpreted uniformly as $1$ or $0$, and likewise for the other connectives. When negation is collapsible, then $\sfrac{1}{2}$ can be thought of as playing the role of $1$ or $0$ across connectives.

\begin{theorem}\label{theorem:st}
Let $\X$ be a three-valued scheme. $\vDash^{st}_{\X} \ = \ \vDash_{2}$ if and only if $\X$ is Boolean normal and either monotonic (see Fig.~\ref{boolean-normal-monotonic}) or collapsible (see Fig.~\ref{boolean-normal-truth-collapsible} and \ref{boolean-normal-falsity-collapsible}).
\end{theorem}

\begin{proof}
See Theorems~\ref{theorem:sufficient-conditions}~and~\ref{theorem:necessary-conditions} in Appendix.
\end{proof}

\bdr{
It follows from
Theorem \ref{thm:ss}, using the $ss$-consequence relation, that there are 8192 different three-valued presentations of classical logic. Similarly, according to Theorem \ref{thm:tt}, we can obtain also 8192 different three-valued presentations of classical logic using the $tt$-consequence relation. Finally, as a consequence of Theorem \ref{theorem:st} there are 528 different three-valued presentations of classical logic, with the $st$-consequence relation.}  

In the next section, we will explore how all of these results can be connected with similar investigations about whether $ss$, $tt$, $st$ and $ts$ and $ss\cap tt$ can support the operational rules of Gentzen's proof system for classical logic.

\section{Comparisons and Perspectives}\label{sec:discussion}

The results of the previous sections tell us, given a three-valued definition of validity, exactly which three-valued truth tables for negation, conjunction, and disjunction, warrant classical inferences for the resulting logic. In subsection \ref{sec:gr}, we compare this finding to results established in \cite{chemla2019many}. In subsection \ref{sect:philosophical} we discuss some philosophical implications of our work regarding the definition of classical logic.  

\subsection{Gentzen-regular connectives}\label{sec:gr}

In \cite{chemla2019many}, a goal partly related to the one discussed here was pursued. Namely, given a three-valued definition of validity, it was asked which three-valued operators are \textit{Gentzen-regular} relative to it. Basically, a Gentzen-regular connective is a connective whose behavior can be characterized in terms of the bidirectional rules of Gentzen's ${\bf LK}$---these rules can therefore be understood as introduction and elimination sequent rules, respectively. For example, the rule whereby $\Gamma, A, B\vdash \Delta$ iff $\Gamma, A\wedge B\vdash \Delta$ corresponds to Gentzen's rule when conjunction occurs in premise position. And the rule whereby $\Gamma \vdash A\wedge B, \Delta$ iff $\Gamma \vdash A, \Delta$ and $\Gamma \vdash B, \Delta$ corresponds to Gentzen's rule for conjunction in conclusion position---see also Figure~\ref{fig:LK-rules}.
We give a more precise definition of Gentzen-regularity in Appendix~\ref{app:gr}, since the definition applies to any $n$-ary connective, beyond negation, conjunction, and disjunction. 

Clearly, when dealing with the usual two-valued semantics for classical logic, all  connectives are Gentzen-regular. However, a consequence relation can fail to be classical at the structural level, but still admit Gentzen-regular connectives. This means that unlike us here, \cite{chemla2019many} did not seek a three-valued characterization of classical logic \textit{qua} combination of operational and structural rules. Instead they focused merely on the operational side of Gentzen's proof system for classical logic.

Furthermore, \cite{chemla2019many}'s results did not seek to characterize {\it schemes} (namely sets of truth tables), but they look at connectives one by one. Consequently, \pe{their} approach is not limited to negation, conjunction and disjunction, or to a particular set of operators, but it covers arbitrary $n$-ary truth-functional operators. However, they assume the language to be constant-expressive, which means that the constants $1$, $0$ and $\sfrac{1}{2}$ are expressible by means of constant symbols.

For comparison, let us consider $st$ and $ts$ first. Under the assumption of constant expressiveness, \cite{chemla2019many} proved that $st$ and $ts$ admit a unique Gentzen-regular negation, a unique Gentzen-regular conjunction, and a unique Gentzen-regular disjunction, described by the Strong Kleene tables. In the case of $ts$, it therefore admits exactly one Gentzen-regular scheme involving negation, disjunction and conjunction. Above we saw that $ts$ admits no trivalent scheme supporting classical logic. There is no contradiction there, since Gentzen-regularity pertains only to the operational rules of a proof system for classical logic, and not to structural rules. This situation may be interpreted by saying that although $ts$ does not support classical inferences, it can support Gentzen-regular connectives that describe, in a way, classical connectives.
 
For $st$, the situation is different: \bdr{as mentioned in the previous section,} Figures \ref{boolean-normal-monotonic} and \ref{boolean-normal-truth-collapsible} and \ref{boolean-normal-falsity-collapsible} together indicate that $st$ admits 528 distinct schemes involving negation, conjunction and disjunction, supporting classical inferences, including the Weak Kleene scheme and more (512 collapsible schemes, and 16 monotone schemes). However, \cite{chemla2019many}'s result implies that $st$ admits a unique Gentzen-regular \textit{scheme}, namely the Strong Kleene one. Whence comes the difference? Here the answer concerns the assumption of constant expressiveness. In \cite{chemla2019many}, Chemla and Egr\'e left as an open issue the characterization of Gentzen-regular three-valued operators when the language does not admit constants for all truth-values. The present inventory can be seen as answering this problem for the case in which the constants are not expressible. 

These comparisons raise the more general question of what may be needed beside the Gentzen-regularity of a connective in order to guarantee that the logic be inferentially classical. The following result gives an answer to this question:

\begin{lemma}
A propositional logic $L=\langle \vdash, C\rangle $ is (inferentially) classical if and only if its connectives in $C$ are Gentzen-regular and $\vdash$ is such that for $\Gamma$ and $\Delta$ two sets of atomic propositions, $\Gamma\vdash \Delta$ iff $\Gamma \cap \Delta \neq \emptyset$.
\label{lemma:gentzen+reflexivity}
\end{lemma}

\begin{proof}
See Appendix \ref{app:gr}.
\end{proof}

A consequence relation like $ts$ obviously fails the structural condition expressed in this lemma, and so cannot support classical logic despite admitting Gentzen-regular connectives. On the other hand, $st$ satisfies the condition, just like $ss$ and $tt$. Finally, while $ss\cap tt$ too satisfies it, it is shown in \cite{chemla2019many} that it does not admit a Gentzen-regular negation.

More generally, we believe the above lemma could be used to answer the question we posed relative to arbitrary connectives beside negation, conjunction and disjunction, drawing on the fact that the notion of Gentzen-regularity can be defined for arbitrary finite operators as discussed in \cite{chemla2019many}.  
We leave this investigation for future work.

\subsection{Philosophical perspectives}\label{sect:philosophical}

The results of this paper show that classical logic can be obtained in a variety of ways in a three-valued setting. This raises the following question: from these various presentations of classical logic, is one of them more fundamental than the others? Besides, aren't all of them just superfluous in comparison to the the standard two-valued presentation of classical logic?

Let us consider $ss$ and $tt$ first. Relative to those systems, Theorems \ref{thm:ss} and \ref{thm:tt} establish that the collapsible schemes support exactly the classical inferences. But they are also schemes in which the middle value mirrors exactly one of the classical values. Hence, this middle value may be judged entirely redundant. We can find instances of this observation in the literature. In \cite{cantwell2008logic}, for instance, Cantwell puts forward a system of trivalent truth tables for negation, conjunction, disjunction, and a conditional operator. 
This system turns out to encapsulate exactly one of the truth-collapsible schemes of Figure \ref{truth-collapsible}, and it is called ``Non-Classical Bivalent'' by Cantwell, precisely because it yields classical logic when paired with $tt$-validity, as presented in \cite[Theorem 4.4]{cantwell2008logic}. In this regard, the interest of Cantwell's conditional operator---proposed earlier by \cite{cooper1968propositional}, see \cite{egre2021finettian} for a comparison---shows up precisely when his conjunction, disjunction and conditional are paired with Strong Kleene negation so as to yield a \textit{noncollapsible, nonclassical} system. 

More generally, the collapsible schemes can be applied a reduction technique on truth values presented as a ``grouping reduction'' in \cite[Appendix C]{chemlaegre2019suszko}, whose goal is precisely to merge truth values that play the same role in premise position and in conclusion positions of arguments. As shown there, for $ss$ and $tt$, grouping reductions basically fulfill Suszko's goal in \cite{suszko1977fregean}: they suggest that an appeal to three truth values is idle when it comes to representing classical inferences in a compositional way, and that two values are all we need.

What about $st$? It was proved that the determination of the minimum number of truth values needed to represent a reflexive, monotonic and transitive consequence relation is exactly two, but that it is three if the relation is reflexive and monotonic but nontransitive (viz. \cite{chemlaegre2019suszko}, Corollary 4.7). But as shown by \cite{barrio2015logics,cobreros2012tolerant,dicher2019st}, $st$ is not a transitive consequence relation. For $st$, therefore, we cannot argue in the same way as $ss$ and $tt$-systems that three values are idle in comparison to using just two values. Besides, as argued by \cite{cobreros2012tolerant} and by \cite{szmucferguson}, the use of a third truth value is independently motivated to represent special semantic status, such as vagueness, or absurdity, or paradoxicality. And for systems of inferences involving sentences with this third semantic status, preserving classical logic for inferences is a conservative benefit.\footnote{Notice, however, that not all the schemes that render classical logic with the $st$ consequence relation are compatible with naive non-trivial theories of truth, vagueness, paradoxicality and so on. In fact, only the monotonic ones are. To wit, consider a Liar sentence $\lambda$ and observe that there can't be a stable valuation for it in a collapsible schema where $v(\lambda) = v(\neg Tr \ulcorner \lambda \urcorner)$, where obviously $Tr$ is a naive truth predicate and $\ulcorner \lambda \urcorner$ is a quotation name for the Liar sentence.}

We can therefore answer the questions raised above as follows: for consequence relations like $ss$ and $tt$, collapsible schemes constitute roundabout ways of representing classical logic compared to the two-valued definition. In the case of $st$, the situation is more complex: while the two-valued approach to classical logic sets a benchmark for the definition of classicality all across the board, we may find different foundations for classicality at the inferential level. At this point, however, more work remains to be done to generalize the present results to more connectives, but also to many-valued logics beyond three values.\footnote{\bdr{This last point has been explored recently by Fitting in \cite{fitting2021family}, \cite{fitting2021strict}. In those papers, the author shows how to build the counterpart of some nonclassical logics using the $st$ consequence relation defined for algebras with more than three values.}} \ds{Furthermore, the present investigations are limited to the propositional case, but one may also be interested in looking at the question of which are all the three-valued presentations of classical logic when such a system is understood as first-order logic. Interestingly enough, the generalizations of the previously discussed results are not always immediate, and the issue is somehow related to the understanding of the universal and existential quantifiers as infinitary versions of conjunction and disjunction, respectively.}\footnote{\ds{For some cases, like the strong Kleene or the weak Kleene schemes, it is well-known and relatively obvious how to devise appropriate quantifiers. This is also true for some Boolean normal collapsible schemes whose operations have only classical outputs. However, both when looking at the Boolean normal monotonic, and the Boolean normal collapsible schemes, there are some (algebraically speaking) asymmetric schemes, where the same pair of inputs gives a certain output in a given order, and another output when considered in the opposite order. For instance, some Boolean normal monotonic schemes are such that $0 \wedge \sfrac{1}{2} = \sfrac{1}{2}$ although $ \sfrac{1}{2} \wedge 0 = 0$. How is one supposed to generalize this asymmetric behavior in order to conceive, e.g., an appropriately infinitary version of this conjunction? It is not obvious whether having a false instance is enough for the quantified statement to be false, or if it's also required that no instances receive the value $\sfrac{1}{2}$. These, and other similar issues, replicate in the case of the other quantifier, as they do for the Boolean normal collapsible schemes.}}

\begin{appendix}
\section{Appendix: proofs}
\subsection{Proofs common to several consequence relations}

\begin{lemma}\label{lemma:bn}
Let $\X$ be a three-valued scheme. If $\X$ is Boolean normal, then $\vDash^{\xy}_{\X} \ \subseteq \ \vDash_{2}$, with $\xy\in\{ss,tt,st\}$.
\end{lemma}

\begin{proof}
\bdr{We need to prove that if $\Gamma \nvDash_{2} \Delta$ then $\Gamma \nvDash^{\xy}_{\mathbf{X}} \Delta$, for every $\Gamma, \Delta$. By a straightforward induction, under the assumption of Boolean normality, it is easy to show that for every classical two-valued valuation $v$, there is a three-valued $\X$ valuation $v^{*}$ such that $v=v^{*}$. Thus given the notions of $ss$-,$tt$- and $st$-consequence relations, if $v$ is a witness of $\Gamma \nvDash_{2} \Delta$, then $v^{*}$ is a witness of $\Gamma \nvDash^{\xy}_{\mathbf{X}} \Delta$.}
\end{proof}

\begin{lemma}\label{lemma:falsitycollapsible}
Let $\mathbf{X}$ be a three-valued scheme. If $\mathbf{X}$ is falsity-collapsible, then $ \vDash_{2} \ \subseteq \ \vDash^{sy}_{\mathbf{X}}$, with $y \in\{s,t\}$.
\end{lemma}

\begin{proof}
Suppose $\Gamma \nvDash^{sy}_{\mathbf{X}} \Delta$, i.e.,either there is a three-valued valuation $v$ such that $v(A)=1$ and $v(B)=0$, for every $A \in \Gamma$ and $B \in \Delta$, if $y=t$, or there is a three-valued valuation $v$ such that $v(A)=1$ and $v(B)\in\{0, \sfrac{1}{2}\}$, for every $A \in \Gamma$ and $B \in \Delta$, if $y=s$. 

Now we will show that in both cases $\Gamma \nvDash_{2} \Delta$, i.e.,that there is a classical two-valued valuation $v^{*}$ such that $v^{*}(A)=1$ and $v^{*}(B)=0$, for every $A \in \Gamma$ and $B \in \Delta$. Consider any of the cases and take $v^{*}$ to be defined as follows:
 
\[
v^{*}(p)=
\begin{cases}
0 & \text{ if } v(p)=\sfrac{1}{2}\\
v(p) & \text{ otherwise}
\end{cases}
\]

Now we show by induction on the complexity of the formula that, on the one hand, if $v(A)=1$, then $v^{*}(A)=1$ and, on the other hand, if $v(A) \in \{0, \sfrac{1}{2}\}$, then $v^{*}(A)=0$.

\underline{Base case:} If $A$ is a propositional letter, then it holds by definition of the valuation $v^{*}$.

\underline{Inductive step:} Here we need to consider three cases:

\begin{itemize}
    \item $A = \neg B$. 
    \begin{itemize}
        \item If $v(\neg B)= 1$ then $v(B)\in \{0,\sfrac{1}{2}\}$. By IH $v^{*}(B)=0$, then $v^{*}(\neg B)=1$.
        
        \item If $v(\neg B)\in\{0,\sfrac{1}{2}\}$ then $v(B)= 1$. By IH $v^{*}(B)=1$, then $v^{*}(\neg B)=0$.
    \end{itemize}

    \item $A = B \wedge C$. 
    \begin{itemize}
        \item If $v(B \wedge C)= 1$ then $v(B)=v(C)=1$. By IH $v^{*}(B)=v^{*}(C)=1$, then $v^{*}(B \wedge C)=1$.
        
        \item If $v(B \wedge C)\in \{0,\sfrac{1}{2}\}$ then $v(B)\in \{0, \sfrac{1}{2}\}$ and $v(C)\in \{0, \sfrac{1}{2}\}$. By IH $v^{*}(B)=v^{*}(C)=0$, and then $v^{*}(B \wedge C)=0$.
    \end{itemize}
    
    \item $A = B \vee C$. 
    \begin{itemize}
        \item If $v(B \vee C)= 1$ then $v(B)=1$ or  $v(C)=1$. So, depending on which of these two is the case, by IH $v^{*}(B)=1$ or $v^{*}(C)=1$, and then $v^{*}(B \vee C)=1$.
        
        \item If $v(B \vee C)\in \{0,\sfrac{1}{2}\}$ then $v(B) \in \{0. \sfrac{1}{2}\}$ and $v(C) \in \{0. \sfrac{1}{2}\}$. By IH $v^{*}(B)=v^{*}(C)=0$, then $v^{*}(B \vee C)=0$.
    \end{itemize}
    \end{itemize}
    
    This shows $v^{*}$ is a classical two-valued valuation witnessing $\Gamma \nvDash_{2} \Delta$, and therefore that $\vDash_{2} \ \subseteq \  \vDash^{st}_{\mathbf{X}}$ as desired. 
\end{proof}

\begin{lemma}\label{lemma:truthcollapsible}
Let $\mathbf{X}$ be a three-valued scheme. If $\mathbf{X}$ is truth-collapsible, then $ \vDash_{2} \ \subseteq \ \vDash^{yt}_{\mathbf{X}}$, with $y \in\{s,t\}$.
\end{lemma}

\begin{proof}
The proof is similar to the previous Lemma, and so we leave it to the reader.
\end{proof}

\subsection{The proofs for $st$}

\begin{lemma}\label{lemma:monotonic}
Let $\mathbf{X}$ be a three-valued scheme. If $\mathbf{X}$ is monotonic, then $ \vDash_{2} \ \subseteq \ \vDash^{st}_{\mathbf{X}}$.
\end{lemma}

\begin{proof}
Assume there is a inference such that  $\Gamma\nvDash^{st}_{\mathbf{X}}\Delta$. Then, there is a valuation $v$, such that $v(A)=1$ for every $A \in \Gamma$ and $v(B)=0$ for every $B \in \Delta$. Now we will show that $\Gamma \nvDash_{2} \Delta$, i.e.,that there is a classical two-valued valuation $v^{*}$ such that $v^{*}(A)=1$ and $v^{*}(B)=0$, for every $A \in \Gamma$ and $B \in \Delta$. We take $v^{*}$ to be defined as follows:
 
\[
v^{*}(p)=
\begin{cases}
0 & \text{ if } v(p)=\sfrac{1}{2}\\
v(p) & \text{ otherwise}
\end{cases}
\]

Now we show by induction on the complexity of the formula that, on the one hand, if $v(A)=1$, then $v^{*}(A)=1$ and, on the other hand, if $v(A)=0$, then $v^{*}(A)=0$.

\underline{Base case:} If $A$ is a propositional letter, then it holds by definition of the valuation $v^{*}$.

\underline{Inductive step:} Here we need to consider three cases:

\begin{itemize}
    \item $A = \neg B$. 
    \begin{itemize}
        \item If $v(\neg B)= 1$ then $v(B)=0$. By IH $v^{*}(B)=0$, then $v^{*}(\neg B)=1$.
        
        \item If $v(\neg B)= 0$ then $v(B)= 1$. By IH $v^{*}(B)=1$, then $v^{*}(\neg B)=0$.
    \end{itemize}

    \item $A = B \wedge C$. 
    \begin{itemize}
        \item If $v(B \wedge C)= 1$ then $v(B)=v(C)=1$. By IH $v^{*}(B)=v^{*}(C)=1$, then $v^{*}(B \wedge C)=1$.
        
        \item If $v(B \wedge C)= 0$ then $v(B)=0$ or $v(C)=0$. Then depending on which of the disjuncts holds, by IH $v^{*}(B)=0$ or $v^{*}(C)=0$, and then $v^{*}(B \wedge C)=0$.
    \end{itemize}
    
    \item $A = B \vee C$. 
    \begin{itemize}
        \item If $v(B \vee C)= 1$ then $v(B)=1$ or $v(C)=1$. So, depending on which of these two is the case, by IH $v^{*}(B)=1$ or $v^{*}(C)=1$, and then $v^{*}(B \vee C)=1$.
        
        \item If $v(B \vee C)= 0$ then $v(B) =0$ and $v(C)=0$. By IH $v^{*}(B)=v^{*}(C)=0$, then $v^{*}(B \vee C)=0$.
    \end{itemize}
    \end{itemize}

    This shows $v^{*}$ is a classical two-valued valuation witnessing $\Gamma \nvDash_{2}\Delta$, and therefore that $\vDash_{2} \ \subseteq \  \vDash^{st}_{\mathbf{X}}$ as desired. 
\end{proof}

\begin{theorem}
\label{theorem:sufficient-conditions}
Let $\mathbf{X}$ be a three-valued scheme. If $\mathbf{X}$ is Boolean normal monotonic, or Boolean normal collapsible, then $\vDash^{st}_{\mathbf{X}} \ = \ \vDash_{2}$.
\end{theorem}

\begin{proof}
From Lemmas \ref{lemma:bn}, \ref{lemma:monotonic}, \ref{lemma:falsitycollapsible} and \ref{lemma:truthcollapsible}.
\end{proof}

Up until now we proved that certain three-valued schemes---belonging in particular into the class of normal Boolean monotonic, or normal Boolean collapsible schemes---render classical logic when equipped with the $st$ definition of logical consequence. If possible, we also would like to prove the converse. That is to say, that if a three-valued scheme renders classical logic when equipped with the $st$ definition of logical consequence, then said scheme belongs in one and only one of the two classes described before. Below, we show this to be the case. However, to prove this we need both some definitions and some important lemmata, that will do all the heavy-lifting for us.

\begin{lemma}\label{lemma:normal}
Let $\mathbf{X}$ be a three-valued scheme. If $\mathbf{X}$ is not Boolean normal, then $\vDash^{st}_{\mathbf{X}} \ \nsubseteq \ \vDash_{2}$.
\end{lemma}

\begin{proof}
Suppose $\mathbf{X}$ is not Boolean normal, then some operation behaves in a way such that some classically invalid inferences are valid in $\mathbf{X}$. 

\begin{enumerate}
\item Let's start with negation. If $\X$ is such that $\neg(1) \in\{\sfrac{1}{2}, 1\}$ then  $ p \vDash^{st}_{\mathbf{X}}\neg p$. On the other hand if $\neg(0)\in\{\sfrac{1}{2}, 0\}$, then $ \neg p \vDash^{st}_{\mathbf{X}} p$.

\item So, having proved that negation must be Boolean normal, if it is $\vee$ which is not Boolean normal, then $ \neg p \vee \neg p \vDash^{st}_{\mathbf{X}} p$, or $\neg p \vee p \vDash^{st}_{\mathbf{X}}  p$, or $ p \vDash^{st}_{\mathbf{X}} \neg p \vee \neg p$

\item Again, knowing that negation is Boolean normal, if it is $\wedge$ which is not Boolean normal, then $ p \vDash^{st}_{\mathbf{X}} p \wedge \neg p$, or $ \neg p \wedge \neg p \vDash^{st}_{\mathbf{X}} p$, or $p \vDash^{st}_{\mathbf{X}}  \neg p \wedge \neg p$
\end{enumerate}

But none of these are valid in classical logic, whence $\vDash^{st}_{\mathbf{X}} \ \nsubseteq \ \vDash_{2}$.
\end{proof}

From this Lemma, since the classical values are determined, we can conclude that there are in principle at most three possible negations to consider: $\neg(\sfrac{1}{2})\in\{0, \sfrac{1}{2}, 1\}$. And actually, what we will prove next is that each of these negations selects exactly the truth tables we have proved are enough to obtain classical logic. In other words, we will prove the following:\\

\begin{lemma}\label{lemma:three-cases}
Let $\X$ be a three-valued scheme. If $\vDash^{st}_{\mathbf{X}} \ = \ \vDash_{2}$ we have three cases:

\begin{description}
    \item[(1)] If $\neg(\sfrac{1}{2})=\sfrac{1}{2}$ then conjunction and disjunction are Boolean normal monotonic (the operations on Fig.~\ref{boolean-normal-monotonic}).
    \item[(2)] If $\neg(\sfrac{1}{2})=0$ then conjunction and disjunction are operations of a Boolean normal truth-collapsible scheme (the operations on Fig.~\ref{boolean-normal-truth-collapsible}).
    \item[(3)] If $\neg(\sfrac{1}{2})=1$ then conjunction and disjunction are operations of a Boolean normal falsity-collapsible scheme (the operations on Fig.~\ref{boolean-normal-falsity-collapsible}).
\end{description}

\end{lemma}

\begin{proof}
By Lemma \ref{lemma:normal} we assume Boolean normality. We will prove cases (1) and (2), since (3) is similar.

\begin{description}
\item[Case (1)] Assume then that $\neg(\sfrac{1}{2})=\sfrac{1}{2}$. We will show that the other operations are monotonic. 

\begin{itemize}
    \item The case of the conjunction:

\begin{itemize}
    \item First we show that in every $\X$, $(\sfrac{1}{2}\wedge \sfrac{1}{2})=\sfrac{1}{2}$.
    
    \begin{itemize}
        \item Assume on the contrary that $(\sfrac{1}{2}\wedge \sfrac{1}{2})=1$. Then we would have a counterexample to the following classically valid inference: $ p \wedge \neg p\nvDash^{st}_{\mathbf{X}} q$ ($v(q)=0$, $v(p)=\sfrac{1}{2}$).

       \item Assume now that $(\sfrac{1}{2}\wedge \sfrac{1}{2})=0$. Then we would have a counterexample to the following classically valid inference: $\neg (\neg p \wedge \neg p) \nvDash^{st}_{\mathbf{X}}  p \wedge p$ ($v(p)=\sfrac{1}{2}$).
    \end{itemize}
    
    \item Now, having proved the previous case, we show that $(\sfrac{1}{2}\wedge 1)=\sfrac{1}{2}$ (we leave to the reader the case $(1\wedge \sfrac{1}{2})=\sfrac{1}{2}$). 
    \begin{itemize}
        \item Assume on the contrary that $(\sfrac{1}{2}\wedge 1)=0$. Then we would have a counterexample to the following classically valid inference: $ p \nvDash^{st}_{\mathbf{X}} \neg (q \wedge \neg q)\wedge p $ ($v(p)=1$, $v(q)=\sfrac{1}{2}$). 
    
       \item Assume now that $(\sfrac{1}{2}\wedge 1)=1$. Then we would have a counterexample to the following classically valid inference: $ (p \wedge \neg p) \wedge q \nvDash^{st}_{\mathbf{X}} \neg q$ ($v(p)=\sfrac{1}{2}$, $v(q)=1$).
    \end{itemize}

\item We show now that $(\sfrac{1}{2}\wedge 0)\neq 1$ (we left to the reader the case $(0\wedge\sfrac{1}{2})\neq 1$). If it were the case that $(\sfrac{1}{2}\wedge 0)= 1$ then we would have a counterexample to the following classically valid inference: $  p \wedge q \nvDash^{st}_{\mathbf{X}} q$ ($v(p)=\sfrac{1}{2}$, $v(q)=0$).
    
\end{itemize}

\item The case of the disjunction:

\begin{itemize}
    \item First we show that in every $\X$, $(\sfrac{1}{2}\vee \sfrac{1}{2})=\sfrac{1}{2}$.
    
    \begin{itemize}
        \item Assume on the contrary that $(\sfrac{1}{2}\vee \sfrac{1}{2})=1$. Then we would have a counterexample to the following classically valid inference: $ p \vee p\nvDash^{st}_{\mathbf{X}} \neg(\neg p \vee \neg p) $ ($v(p)=\sfrac{1}{2}$). 
    
       \item Assume now that $(\sfrac{1}{2}\vee \sfrac{1}{2})=0$. Then we would have a counterexample to the following classically valid inference: $\nvDash^{st}_{\mathbf{X}}p \vee \neg p$ ($v(p)=\sfrac{1}{2}$).
    \end{itemize}
    
    \item Now, having proved the previous case, we show that $(\sfrac{1}{2}\vee 1)\neq 0$ (we left to the reader the case $(1\vee \sfrac{1}{2})\neq\sfrac{1}{2})$). If it were the case that $(\sfrac{1}{2}\vee 1)= 0$ then we would have a counterexample to the following classically valid inference: $q \nvDash^{st}_{\mathbf{X}}  p \vee q$ ($v(p)=\sfrac{1}{2}$, $v(q)=1$).
    
    \item We show now that $(\sfrac{1}{2}\vee 0)=\sfrac{1}{2}$ (we left to the reader the case $(0\vee\sfrac{1}{2})=\sfrac{1}{2}$). 
    
     \begin{itemize}
        \item Assume on the contrary that $(\sfrac{1}{2}\vee 0)=1$. Then we would have a counterexample to the following classically valid inference: $\neg (p \vee \neg p) \vee q \nvDash^{st}_{\mathbf{X}}  q $ ($v(p)=\sfrac{1}{2}$, $v(q)=0$). 
    
       \item Assume now that $(\sfrac{1}{2}\vee 0)=0$. Then we would have a counterexample to the following classically valid inference: $ \nvDash^{st}_{\mathbf{X}}  (p \vee \neg p) \vee q$ ($v(p)=\sfrac{1}{2}$, $v(q)=0$).
    \end{itemize}

\end{itemize}

\end{itemize}
\item[Case (2)] Assume now that $\neg(\sfrac{1}{2})=0$. We will show that the other operations belong to some of the truth-collapsible schemes. 

\begin{itemize}
    \item The case of the conjunction:

\begin{itemize}
    \item First we show that $(1\wedge \sfrac{1}{2})\neq 0$ (we left to the reader the case $(\sfrac{1}{2}\wedge 1)\neq 0$). Assume on the contrary that $(1\wedge \sfrac{1}{2})=0$. Then we would have a counterexample to the following classically valid inference: $\neg (p \wedge q), p \nvDash^{st}_{\mathbf{X}} \neg q$ ($v(p)=1$, $v(q)=\sfrac{1}{2}$). 
    
    \item Now, we show that $(\sfrac{1}{2}\wedge \sfrac{1}{2})\neq 0$. Assume on the contrary that $(\sfrac{1}{2}\wedge \sfrac{1}{2})=0$. Then we would have a counterexample to the following classically valid inference: $ \neg \neg p \nvDash^{st}_{\mathbf{X}} p \wedge p$ ($v(p)=\sfrac{1}{2}$).

\item We show now that $(\sfrac{1}{2}\wedge 0)= 0$ (we left to the reader the case $(0\wedge\sfrac{1}{2})= 0$). 

     \begin{itemize}
        \item Assume on the contrary that $(\sfrac{1}{2}\wedge 0)=1$. Then we would have a counterexample to the following classically valid inference: $ p \wedge \neg p \nvDash^{st}_{\mathbf{X}} q $ ($v(p)=\sfrac{1}{2}$, $v(q)=0$). 
    
       \item Assume now that $(\sfrac{1}{2}\wedge 0)=\sfrac{1}{2}$. Then we would have a counterexample to the following classically valid inference: $\nvDash^{st}_{\mathbf{X}} \neg (p \wedge \neg p)$ ($v(p)=\sfrac{1}{2}$).
    \end{itemize}
\end{itemize}

\item The case of the disjunction:

\begin{itemize}
    \item First we show that $(1\vee \sfrac{1}{2})\neq 0$ (we left to the reader the case $(\sfrac{1}{2}\vee 1)\neq 0$). Assume on the contrary that $(1\vee \sfrac{1}{2})=0$. Then we would have a counterexample to the following classically valid inference: $p \nvDash^{st}_{\mathbf{X}} p \vee q$ ($v(p)=1$, $v(q)=\sfrac{1}{2}$). 
    
    \item Now, we show that $(\sfrac{1}{2}\vee \sfrac{1}{2})\neq 0$. Assume on the contrary that $(\sfrac{1}{2}\vee \sfrac{1}{2})=0$. Then we would have a counterexample to the following classically valid inference: $\neg \neg p \nvDash^{st}_{\mathbf{X}} p \vee p$ ($v(p)=\sfrac{1}{2}$).

\item We show now that $(\sfrac{1}{2}\vee 0)\neq 0$ (we left to the reader the case $(0\vee\sfrac{1}{2})\neq 0$). Assume on the contrary that $(\sfrac{1}{2}\vee 0)=0$. Then we would have a counterexample to the following classically valid inference: $\nvDash^{st}_{\mathbf{X}}  p \vee \neg p$ ($v(p)=\sfrac{1}{2}$). 

\end{itemize}

\end{itemize}

\item[Case (3)] Similar to the Case (2), so we leave it to the reader.\qedhere
\end{description}
\end{proof}

\begin{theorem}
\label{theorem:necessary-conditions}
Let $\mathbf{X}$ be a three-valued scheme. If $\vDash^{st}_{\mathbf{X}} \ = \ \vDash_{2}$, then $\mathbf{X}$ is Boolean normal, and either monotonic, or collapsible.

\end{theorem}

\begin{proof}
Immediate from Lemmas \ref{lemma:normal} and \ref{lemma:three-cases}.
\end{proof}

\subsection{The proofs for $ss$}

\begin{lemma}\label{lemma:falsity-collapsible and ss}
Let $\mathbf{X}$ be a three-valued scheme. If $\mathbf{X}$ is falsity-collapsible, then $\vDash^{ss}_{\mathbf{X}} \ \subseteq \ \vDash_{2}$.
\end{lemma}

\begin{proof}

\bdr{Suppose $\Gamma \nvDash_{2} \Delta$, i.e.,there is a two-valued valuation $v$ such that $v(A)=1$ and $v(B)=0$, for every $A \in \Gamma$ and $B \in \Delta$. Now we will show that $\Gamma \nvDash^{ss}_{\mathbf{X}}\Delta$, i.e.,that there is a three-valued valuation $v^{*}$ such that $v^{*}(A)=1$ and $v^{*}(B)\in\{\sfrac{1}{2},0\}$, for every $A \in \Gamma$ and $B \in \Delta$. We take $v^{*}$ to be defined as follows:

\[
v^{*}(p)=
v(p)
\]

Now we show by induction on the complexity of the formula that, on the one hand, if $v(A)=0$, then $v^{*}(A)\in\{\sfrac{1}{2},0\}$ and, on the other hand, if $v(A)=1$, then $v^{*}(A)=1$.

\underline{Base case:} If $A$ is a propositional letter, then it holds by definition of the valuation $v^{*}$.

\underline{Inductive step:} Here we need to consider three cases:

\begin{itemize}
    \item $A = \neg B$. 
    \begin{itemize}
        \item If $v(\neg B)= 0$ then $v(B)=1$. By IH $v^{*}(B)=1$, then since $\mathbf{X}$ is falsity-collapsible $v^{*}(\neg B)\in \{0,\sfrac{1}{2}\}$.
        
        \item If $v(\neg B)= 1$ then $v(B)= 0$. By IH $v^{*}(B)\in\{\sfrac{1}{2},0\}$, then since $\mathbf{X}$ is falsity-collapsible $v^{*}(\neg B)=1$.

    \end{itemize}

    \item $A = B \wedge C$. 
    \begin{itemize}
        \item If $v(B \wedge C)= 0$ then $v(B)=0$ or $v(C)=0$. By IH $v^{*}(B)\in\{\sfrac{1}{2},0\}$ or $v^{*}(C)\in\{\sfrac{1}{2},0\}$, then since $\mathbf{X}$ is falsity-collapsible $v^{*}(B \wedge C)\in \{0,\sfrac{1}{2}\}$.
        
        \item If $v(B \wedge C)= 1$ then $v(B)= v(C)= 1$. Then, by IH $v^{*}(B)= v^{*}(C)= 1$. Thus, since $\mathbf{X}$ is falsity-collapsible $v^{*}(B \wedge C)=1$.
        
    \end{itemize}
    
    \item $A = B \vee C$. 
    \begin{itemize}
        \item If $v(B \vee C)= 1$ then $v(B)=1$ or  $v(C)=1$. So, depending on which of these two is the case, by IH $v^{*}(B)=1$ or $v^{*}(C)=1$, and then since $\mathbf{X}$ is falsity-collapsible $v^{*}(B \vee C)=1$.
        
        \item If $v(B \vee C)= 0$ then $v(B)=v(C)=0$. By IH $v^{*}(B)\in\{\sfrac{1}{2},0\}$ and $v^{*}(C)\in\{\sfrac{1}{2},0\}$, then since $\mathbf{X}$ is falsity-collapsible $v^{*}(B \vee C)\in \{0,\sfrac{1}{2}\}$.
        
    \end{itemize}
    \end{itemize}

    This shows $v^{*}$ is a three-valued valuation witnessing $\Gamma \nvDash^{ss}_{\mathbf{X}} \Delta$, and therefore that $ \vDash^{ss}_{\mathbf{X}} \ \subseteq \ \vDash_{2} $ as desired. 
    
}

    \end{proof}

\begin{theorem}\label{theorem:sufficient-conditions-ss}
Let $\mathbf{X}$ be a three-valued scheme. If $\mathbf{X}$ is falsity-collapsible, then $\vDash^{ss}_{\mathbf{X}} \ = \ \vDash_{2}$.

\end{theorem}

\begin{proof}
From Lemmas \ref{lemma:falsitycollapsible} and \ref{lemma:falsity-collapsible and ss}
\end{proof}

\begin{theorem}
\label{theorem:necessary-conditions-ss}
Let $\mathbf{X}$ be a three-valued scheme. If $\vDash^{ss}_{\mathbf{X}} \ = \ \vDash_{2}$, then $\mathbf{X}$ is falsity-collapsible (i.e.,the operations are those of the schemes in Figure \ref{falsity-collapsible}).
\end{theorem}

\begin{proof}
We will show that if a three-valued scheme $\mathbf{X}$ is not falsity-collapsible then $\ \vDash_{2}\ \nsubseteq \vDash^{ss}_{\mathbf{X}} $. We will show it by cases, considering in order each of the connectives of each possible non-falsity-collapsible scheme. 

\begin{itemize}
    \item The case of the negation: 
    \begin{enumerate}
        \item Assume $\mathbf{X}$ is such that $\neg 1 = 1$. Then, $ p, \neg p \nvDash^{ss}_{\mathbf{X}} q$, but of course $p, \neg p \vDash_{2}  q$.
        \item Assume $\mathbf{X}$ is such that $\neg \sfrac{1}{2} \neq 1$. Then, $ q \nvDash^{ss}_{\mathbf{X}} p, \neg p$, but of course $q \vDash_{2}  p, \neg p$.
        \item Assume $\mathbf{X}$ is such that $\neg 0 \neq 1$. Then, $ q \nvDash^{ss}_{\mathbf{X}} p, \neg p$, but of course $q \vDash_{2}  p, \neg p$.
    \end{enumerate}
    
     \item The case of the conjunction: 
    \begin{enumerate}
        \item Assume $\mathbf{X}$ is such that $x \wedge y \neq 1$, for $x=1$ and $y=1$. Then, $ p, q \nvDash^{ss}_{\mathbf{X}}p \wedge q$, but of course $ p, q \vDash_{2}p \wedge q$.
         \item Assume $\mathbf{X}$ is such that $x \wedge y = 1$, for $x\neq 1$. Then, $ p \wedge q \nvDash^{ss}_{\mathbf{X}} p$, but of course $ p \wedge q \vDash_{2} p$.
         \item Assume $\mathbf{X}$ is such that $x \wedge y = 1$, for $y\neq 1$. Then, $ p \wedge q \nvDash^{ss}_{\mathbf{X}} q$, but of course $ p \wedge q \vDash_{2} q$.
    \end{enumerate}
    
     \item The case of the disjunction: 
    \begin{enumerate}
        \item Assume $\mathbf{X}$ is such that $x \vee y \neq 1$, for $x= 1$. Then, $ p \nvDash^{ss}_{\mathbf{X}} p \vee q$, but of course $ p \vDash_{2} p \vee q$.
        \item Assume $\mathbf{X}$ is such that $x \vee y \neq 1$, for $y= 1$. Then, $ q \nvDash^{ss}_{\mathbf{X}} p \vee q$, but of course $ q \vDash_{2}p \vee q$.
         \item Assume $\mathbf{X}$ is such that $x \vee y = 1$, for $x \neq 1$ and $y \neq 1$. Then, $ p \vee q \nvDash^{ss}_{\mathbf{X}} p, q$, but of course $ p \vee q \vDash_{2} p, q$. \qedhere
    \end{enumerate}
\end{itemize}
\end{proof}

\subsection{The proofs for $tt$}

We omit all the proofs of this section, since basically they are dual to those for $ss$.

\begin{lemma}\label{lemma:truth-collapsible and tt}
Let $\mathbf{X}$ be a three-valued scheme. If $\mathbf{X}$ is truth-collapsible, then $\vDash^{tt}_{\mathbf{X}} \ \subseteq \ \vDash_{2}$.
\end{lemma}

\begin{theorem}
\label{theorem:sufficient-conditions-tt}
Let $\mathbf{X}$ be a three-valued scheme. If $\mathbf{X}$ is truth-collapsible, then $\vDash^{tt}_{\mathbf{X}} \ = \ \vDash_{2}$.
\end{theorem}

\begin{proof}
From Lemmas \ref{lemma:truth-collapsible and tt} and \ref{lemma:truthcollapsible}.
\end{proof}

\begin{theorem}
\label{theorem:necessary-conditions-tt}
Let $\mathbf{X}$ be a three-valued scheme. If $\vDash^{tt}_{\mathbf{X}} \ = \ \vDash_{2}$, then $\mathbf{X}$ is truth-collapsible (i.e.,the operations are those of the schemes in Figure \ref{truth-collapsible}).
\end{theorem}

\section{Gentzen-regularity and classical logic}\label{app:gr}

Following \cite{chemla2019many}, we call a connective Gentzen-regular if its behavior, whether in the conclusion or in the premise of an argument, can be explained fully in terms of conjunction of sequents involving the subformulae related by that connective. Formally, the definition is the following:

\begin{definition}[Gentzen-regular connectives]\label{def:regconn}
	Given a consequence relation $\vdash$, an $n$-ary connective $C$ (for $n\geq 0$) is \textit{Gentzen-regular} for it if there exist
	$\mathcal{B}^p\subseteq\mathcal{P}(\{1,..., n\})\times\mathcal{P}(\{1,..., n\})$ and
	$\mathcal{B}^c\subseteq\mathcal{P}(\{1,..., n\})\times\mathcal{P}(\{1,..., n\})$
such that
	$\forall\Gamma, \Delta, \forall F_1, ..., F_n:$
	\[\begin{array}{c@{\textrm{ iff }}c}
	\Gamma, C(F_1, ..., F_n) \vdash \Delta
		& \bigwedge\limits_{(B_p,B_c)\in \mathcal{B}^p} 
			{\Gamma, \{F_i: i\in B_p\}\vdash \{F_i: i\in B_c\}, \Delta}\\

	\Gamma \vdash C(F_1, ..., F_n) , \Delta
		& \bigwedge\limits_{(B_p,B_c)\in \mathcal{B}^c} 
			{\Gamma, \{F_i: i\in B_p\}\vdash \{F_i: i\in B_c\}, \Delta}\\
	\end{array}\]
	\end{definition}
	
	The next lemma relates this feature of the connectives, what it is for a logic to be classical, and a structural condition on sets of atomic propositions (atom-sharing between premises and conclusions). 

\begin{lemma}\label{lem:gr}
A propositional logic $L=\langle \pe{\mathcal{L}}, \vdash, C\rangle $ is {inferentially} classical if and only if its connectives in $C$ are Gentzen-regular and $\vdash$ is such that for $\Gamma$ and $\Delta$ any two sets of atomic propositions, $\Gamma\vdash \Delta$ iff $\Gamma \cap \Delta \neq \emptyset$.
\end{lemma}

\begin{proof}
The left-to-right direction holds because in classical logic, and in any logic that satisfies the same inferences, connectives are Gentzen-regular, and inferences involving only atomic propositions behave as described.

Conversely, suppose the right-hand-side holds for a logic $L$. Consider then sets of premises and conclusions $\Gamma$ and $\Delta$. If $\Gamma$ and $\Delta$ only contain atomic propositions, then $\Gamma\vdash \Delta$ holds in $L$ iff $\Gamma \cap \Delta \neq \emptyset$, by hypothesis, iff it holds in classical logic then.

By induction on the complexity of the formulae involved, the assumption \pe{of Gentzen-regularity} allows us to generalize this equivalence between $L$ and classical logic to inferences with non-atomic propositions. Indeed, the Gentzen regularity rules reduce the validity of any inference $\Gamma\vdash \Delta$ to the validity of a conjunction of inferences involving formulae of strictly lower syntactic complexity.\footnote{One limit case may be mentioned: Gentzen-regularity rules may reduce the complexity so much that they eliminate the formula altogether. You may obtain this through an empty conjunction in the Definition~\ref{def:regconn}. As an illustration, $\top$ seen as a $0$-ary connective, has such a rule for its  Gentzen-conclusion-rule: $\Gamma \vdash \top, \Delta$ is valid no matter what. This edge case does not block the inductive step of this proof.}



For concreteness, consider $\Gamma', A\vee B\vdash \Delta$. This holds if and only if both $\Gamma', A\vdash \Delta$ and $\Gamma', B\vdash \Delta$ hold. This shows how the Gentzen premise-rule reduces the verification of inferences with disjunction in premises, to the verification of strictly simpler inferences, with no disjunctions in premises.
Eliminating connectives one after the other thanks to Gentzen rules, in premises and in conclusion, we can recursively reduce the complexity of the inferences \pe{until no more reduction is possible}. That is, we can find sets of \textit{atomic} propositions $\Gamma_i, \Delta_i$ such that $\Gamma\vdash \Delta$ holds if and only if the conjunction of the $\Gamma_i\vdash \Delta_i$ hold. 
\end{proof}

Three remarks may be made about this result. The first is that a logic can obey the conditions of Lemma \ref{lem:gr} without coinciding exactly with classical logic. For instance, if $C=\{\neg, \leftrightarrow\}$, with $\neg$ and $\leftrightarrow$ obeying the expected Gentzen rules, then the resulting logic is inferentially classical but is only a fragment of classical logic (because it is functionally incomplete). The second is that irrespective of how Gentzen regular connectives are \textit{named} in $\mathcal{L}$, what matters concerns which operations they correspond to. To use the same example as in the proof, if $\Gamma', A\wedge B\vdash \Delta$ holds iff $\Gamma', A\vdash \Delta$ and $\Gamma', B\vdash \Delta$ hold in $\mathcal{L}$, then it means that ``$\wedge$'' is actually just another name for disjunction in that logic. The third finally is that the structural condition in the Lemma simply corresponds to a form of (strong) Reflexivity on the atoms. It can be verified that it directly implies the admissibility of other classical structural rules, including Exchange, Contraction, Weakening, and Cut, for sequents involving only atoms. For example, if $\Gamma, \Gamma', \Delta, \Delta'$ are sets of atoms, then it follows that $\Gamma\vdash \Delta, p$ and $\Gamma', p\vdash \Delta'$ imply $\Gamma, \Gamma'\vdash \Delta, \Delta'$.

\section{Monotonic operators}\label{app:mon}

\begin{fact}\label{fact:mon}
Given a truth table for a unary or binary operator $f$, the operator is monotonic only if no two horizontally or vertically adjacent cells of the corresponding matrix contain a 1 and a 0.
\end{fact}

\begin{proof}

In the unary case, suppose as a particular case that $f(1/2)=0$ and $f(1)=1$. Then, although $1/2<_{\rm I}1$, their images by $f$ are incomparable. The other cases are symmetric. In the binary case, suppose as a particular case of vertical adjacent cells that $f(1/2,1)=0$ when $f(0,1)=1$. Then although $(1/2,1)<^{comp}_{\rm I}(0,1)$, their images by $f$ are incomparable, which violates monotonicity. The other cases are symmetric.
\end{proof}

\begin{fact}

A binary normal Boolean operator $f$ is monotonic if and only if no two adjacent cells of its matrix get values $1$ and $0$ and if $f(1/2,1/2)$ is not greater than or incomparable with the value of any other cell.

\end{fact}

\begin{proof} From left-to-right, suppose that $f(1/2,1/2)$ is incomparable with or greater than the value of some other cell. Since $(1/2,1/2)<^{comp}_{\rm I}(x,y)$ for all other cells $(x,y)$, this violates monotonicity. The other condition is entailed by Fact \ref{fact:mon}. 

From right-to-left, suppose that $f$ is normal but not monotonic. Then there exist $(x,y)\leq^{comp}_{\rm I} (x',y')$, but either $f(x,y)\geq^{comp}_{\rm I}f(x',y')$, or $f(x,y)$ and $f(x',y')$ are incomparable. If $(x,y)$ is of type $(c,1/2)$ or $(1/2,c)$ with $c$ classical, and $(x',y')$ of type $(c,c)$, then necessarily one of them is 1 and the other 0. If $(x,y)$ is $(1/2, 1/2)$, then the violation is necessary because $f(x,y)=c$ and $f(x',y')$ is $1/2$ or incomparable.
\end{proof}

\end{appendix}

\bibliographystyle{abbrv}
\bibliography{references}

\end{document}